\documentclass[12pt]{article}

\usepackage{amsfonts, amssymb,  amsthm, amsmath}
\usepackage{fancyhdr}

\usepackage{tikz}
\usepackage[all]{xy}

\usetikzlibrary{arrows}

\usepackage{color,hyperref}
    \catcode`\_=11\relax
    \newcommand\email[1]{\_email #1\q_nil}
    \def\_email#1@#2\q_nil{%
      \href{mailto:#1@#2}{{\emailfont #1\emailampersat #2}}
    }
    \newcommand\emailfont{\sffamily}
    \newcommand\emailampersat{{\color{red}\small@}}
    \catcode`\_=8\relax

\theoremstyle{plain}
\newtheorem{thm}{Theorem}[section]

\newtheorem{lemma}[thm]{Lemma}

\theoremstyle{definition}
\newtheorem{definition}[thm]{Definition}

\newtheorem{example}[thm]{Example}

\newcommand{\be}{\begin{equation} } 
\newcommand{\ee}{\end{equation}}

\newcommand{\sa}{\Sigma}

\newcommand{\G}{\mathcal{G}}

\newcommand{\M}{\mathcal{M}eas_T}

\def\@normalsize{\@setsize\normalsize{14.5pt}\xiipt\@xiipt
\abovedisplayskip 12\p@ plus3\p@ minus7\p@
\belowdisplayskip \abovedisplayskip
\abovedisplayshortskip  \z@ plus3\p@
\belowdisplayshortskip  6.5\p@ plus3.5\p@ minus3\p@
\let\@listi\@listI}
\def\keywords{\vspace{.3em}
{\textit{Keywords}: Quantifiers in Probability,  First-order probability logic, Giry Monad \relax%
}}

\def\classification{\vspace{.5em}
{\textit{2010 MSC:  Primary} 03B48, 60A05, 18B99; \textit{ Secondary} 18C20, 18B10 \relax%
}}

\title{Quantifiers as Adjoints in Probability}

\author{Kirk Sturtz}
\date{}
 
\fancypagestyle{plain}{
\fancyhf{} 
\fancyfoot[C]{\thepage} 
  
}

\numberwithin{equation}{section}

\begin{document}
 \maketitle
 \begin{abstract} 
Using the Kleisi category of the Giry monad
the deterministic existential and universal quantifiers are generalized to incorporate nondeterminism.    These probabilistic quantifiers are quantified over the points of the  category  which are probability measures.  
\begin{flushleft}
\keywords 
\classification
\end{flushleft}
\end{abstract}

 \thispagestyle{empty}
 
\section{Introduction}  Probability theory is the logic of propositions under uncertainty
where deductive reasoning is not possible. 
Because predicates can be viewed as a family of propositions parameterized over some set, a better characterization is that probability theory is the logic of predicates under uncertainty.   In estimation problems such predicates are often specified  as a probability density function defined on the parameter space.
 The main obstacle to the further development of this ``predicate logic under uncertainty'',  to lay claim to a first-order logic under uncertainty, has been the treatment of quantifiers.  In the 1960's the main obstacle to the further development of algebraic logic was also the treatment of quantifiers.  Categorical logic solved this problem elegantly with the recognition that quantifiers are adjoint functors.  
  This  ``quantifiers as adjoints''  paradigm can be applied to the Kleisi category of the Giry monad,  denoted $\M$, whose objects are measurable spaces and  arrows $X \rightarrow Y$ are  measurable functions $X \rightarrow TY$ assigning a probability measure on $Y$ for each $x \in X$.  By conceptualizing quantifiers as adjoints the deterministic existential and universal  quantifiers are  extended  to include nondeterminism.  
    
The term ``probabilistic''  is used to mean either deterministic or  nondeterministic where the latter two terms have a precise definition stated in section 2.  
Under this convention our purpose is to generalize the deterministic existential and universal quantifiers to  probabilistic quantifiers for the category $\M$.

In the literature probability quantifiers have been studied almost exclusively from the model theoretic point of view~\cite{Hoover,Keisler1,Ras}.  In these models the quantification \mbox{$(P\textbf{x}\ge r)(\psi(\textbf{x}))$} refers to a set $\{ \textbf{x} |  \psi(\textbf{x}) \}$ having probability measure at least $r$.  The existential and universal quantifiers introduced here are fundamentally different in that these quantifiers do not require the condition $\psi$ defining the subset to be deterministic, and   
 these probabilistic  quantifiers quantify over probability measures.   From the categorical perspective quantification over the points of the category is the appropriate generalization of the deterministic concept, and the points of the category $\M$ are probability measures.

 This paper is organized as follows: Section 2  defines and gives the basic properties of the category $\M$.  Section 3 reviews deterministic quantifiers as adjoints in the category $Sets$.  Section 4, where new results are presented,  then extends the deterministic quantifiers to the category $\M$ giving probabilistic quantifiers.

\section{The Kleisi Category of the Giry Monad}
We denote the category of measurable spaces by $\mathcal{M}eas$.  Given an object $(X,\sa_X)$ in $\mathcal{M}eas$ we often denote it simply by $X$.
 The Giry monad\footnote{The theory of monads can be found in the text~\cite{BW,MacLane}.  The Giry monad is given in \cite{Giry}.} is the endofunctor $T$ defined on  $\mathcal{M}eas$ by the function defined on objects $X$ by
 \begin{equation*}
 \begin{array}{lcl} 
 TX &=& \textrm{The set of probability measures on }X \textrm{ endowed with the initial }\sigma-\textrm{algebra} \\
 && \sa_{TX}\textrm{ generated by the evaluation maps } \\ 
 && \quad \quad \quad \quad \quad \quad  \quad  \begin{array}{lcccc} ev_A&:&TX &\rightarrow & [0,1] \\
 &:& P & \mapsto & P[A]
 \end{array} \\
&&\textrm{for } A \in \sa_X.
 \end{array}
 \end{equation*}
 and the function defined on measurable functions $X \stackrel{f}{\longrightarrow} Y$ by $T(f)P = Pf^{-1}[\cdot] \in TY$.
 The unit of the monad $T$ is the measurable function $X \stackrel{\eta_X}{\longrightarrow} TX$ given by $\eta_X(x) = \delta_x$, the dirac measure on $X$ at $x$.
 The multiplication of the monad is $T^2X \stackrel{\mu_X}{\longrightarrow} TX$ defined at $Q \in T^2X$ by the probability measure $\mu_X(Q)$ on $TX$ at $A \in \sa_X$ by
 \begin{equation*}
 \mu_X(Q)[A] = \int_{q \in TX} q[A] \, dQ.
 \end{equation*}
 
 The Kleisi category of the Giry monad, denoted $\M$, thus has arrows $X \rightarrow Y$ denoting a measurable function $X \rightarrow TY$ and the composition $X \stackrel{f}{\longrightarrow} Y \stackrel{g}{\longrightarrow} Z$ in $\M$, denoted $g \bullet f$,  is the composite of the measurable functions 
 \begin{equation*}
 \begin{tikzpicture}[baseline=(current bounding box.center)]
 
 	\node	(X)	at	(0,0)		         {$X$};
	\node	(TY)	at	(2,0)	         {$TY$};
	\node	(T2Z)  at	(4,0)          {$T^2Z$};
	\node         (TZ)    at  (6,0)                 {$TZ$};
	\node        (cat)   at       (8,0)        {in $\mathcal{M}eas$};
	\draw[->, above] (X) to node  {$f$} (TY);
	\draw[->, above] (TY) to node {$Tg$} (T2Z);
	\draw[->,above]  (T2Z)  to node {$\mu_Z$}  (TZ);
 \end{tikzpicture}
 \end{equation*}
with the composition in $\mathcal{M}eas$ being denoted as $\mu_Z \circ Tg \circ f$.  The context will make clear whether an arrow $f$ is to be interpreted as an arrow of $\M$ or $\mathcal{M}eas$.
 
  The image of a point $x\in X$ under the composite $g \bullet f$ is the probability measure on $Z$ defined at $C \in \sa_Z$ by
 \begin{equation*} 
 \int_Y g(y)[C] \, d(f(x)).
 \end{equation*}
 The identity arrows in $\M$ are given by the unit $\eta_X$ of the monad for each object $X$ in $\M$.  We denote these identity arrows in $\M$ by $\eta$ or $1$ rather than $Id$ because the measurable functions $TX \stackrel{Id_{TX}}{\longrightarrow} TX$ play a significant role as  arrows $TX \stackrel{Id}{\longrightarrow} X$  in $\M$.   This significance come from the observation that by treating $P \in TX$ as the variable argument it follows that for each $A \in \sa_X$  the evaluation map $ev_A(\cdot): TX \rightarrow [0,1]$ satsifies $ev_A(P) = Id(P)[A]$.  Thus the $\sigma$-algebra of $TX$ is that induced by the $\M$ arrow $Id_{TX}$.
 
 \paragraph{Deterministic and Nondeterministic Maps}
   Every measurable mapping in the category of measurable spaces, 
\begin{equation*} 
X  \stackrel{f}{\longrightarrow} Y \quad \textrm{in } \mathcal{M}eas
\end{equation*}
may be regarded as a $\M$ arrow
\begin{equation*}
 X  \stackrel{\delta_f}{\longrightarrow} Y \quad  \textrm{in  } \M
\end{equation*}
defined by the dirac (or one point) measure 
\begin{center}
\begin{equation*}
\delta_f(x)[B] = 
 \left\{ \begin{array}{ll} 1 & \textrm{iff }f(x)\in B
\\ 0 & \textrm{otherwise} \end{array} \right..
\end{equation*}
\end{center}
Thus $\delta_f$ assigns to $x$ the  dirac
measure on $Y$ which is concentrated at $f(x)$.  

\begin{definition}  A  $\M$ arrow $X \stackrel{P}{\longrightarrow} Y$  is \emph{deterministic} if and only if for every $B \in \sa_Y$ the measurable functions $X \stackrel{P_B}{\longrightarrow} [0,1]$ defined by $P_B(x)=P(x)[B]$ assumes a value of either $0$ or $1$.
\end{definition}  \index{deterministic}

Any $\M$ arrow which is not deterministic is called nondeterministic.  The following result provides a useful  characterization of  deterministic mappings.

\begin{thm}  \label{thm:deterministic} Let $(Y,\sa_Y)$ be a countably generated space.  A $\M$ arrow $X \stackrel{P}{\longrightarrow} Y$ is deterministic if and only if it is determined by a measurable function $f$ so $P=\delta_f$.
\end{thm}
\begin{proof}  If $P=\delta_f$ for some measurable function then it assumes a value of either $0$ or $1$ so is deterministic.  To prove the converse let $\tilde{ \mathcal{G}}$ be a countable generating set for $\sa_Y$.  Then \mbox{$ \mathcal{G} = \{ G | G \in \tilde{ \mathcal{G}} \textrm{ or }G^c \in \tilde{ \mathcal{G}}\} \cup \{Y, \emptyset\}$} is a countable generating set.  Suppose for each $B \in \mathcal{G}$ that $P_B\stackrel{\triangle}{=} P(\cdot)[B]$ assumes a value of either $0$ or $1$ on $X$.
We claim there exist a measurable function $f:X \rightarrow Y$ such that  the diagram 

 \begin{equation*} 
 \begin{tikzpicture}[baseline=(current bounding box.center)]
 
 	\node	(X)	at	(0,0)		         {$X$};
	\node	(01)	at	(2,0)	         {$[0,1]$};
	\node	(Y)	at	(2,-1.5)          {$Y$};
	\node        (M)   at       (4,-.75)        {in $\mathcal{M}eas$};
	\draw[->, right, auto] (X) to node  {$P_B$} (01);
	\draw[->, below] (X) to node {$f$} (Y);
	\draw[->,right]  (Y)  to node {$\chi_B$}  (01);

 \end{tikzpicture}
 \end{equation*}
commutes, where $\chi_B$ is the characteristic function.
 
 For each $x \in X$  let $x_1 = \{B \in \G \, | \, P_B(x)=1 \}$ and  $x_0 = \{B \in \G \, | \, P_B(x)=0 \}$.  If $B \in x_1$ then the condition $P_{B} = \chi_{B} \circ f$ requires  that $f:X \rightarrow Y$  satisfies \mbox{$f(x) \in \cap  x_1$}.  
We claim that $\cap  x_1 \ne \emptyset$.    If $\cap x_1 = \emptyset$ then $\cup x_0 =  Y$.  Since $\G$ is countable this implies there exists a countable covering of $Y$ by measurable sets in $x_0$ which can be used to generate a countable disjoint cover $\{C_i\}_{i \in \mathbb{N}}$ of $Y$ by measurable sets $C_i \in \sa_Y$ satisfying $P(x)[C_i]=0$.  By countable additivity of the measure $P(x)$ this implies $P(x)[\cup_{i \in \mathbb{N}} C_i] = P(x)[Y]=0$ which contradicts the required hypothesis of $P(x)[Y]=1$.
Thus $\cap x_1 \ne \emptyset$ and  for each \mbox{$x \in X$} we can choose any set function $f$ satisfying \mbox{$f(x) \in \cap x_1$}.   Then the condition \mbox{$P_{B} = \chi_{B} \circ f$} holds for all $x \in X$ and all $B \in  \mathcal{G}$. 

 It  remains to prove that $f$ is measurable.   
Since $\{1\} \in \mathcal{B}_{[0,1]}$ we have \mbox{$P_{B}^{-1}( \{1\}) \in \Sigma_X$} and so
\mbox{$P_{B}^{-1}(\{1\})= f^{-1}(\chi_{B}^{-1}(\{1\})=f^{-1}(B) \in \Sigma_X$}.  Hence $f$ is measurable.   
\end{proof}

 \paragraph{Predicates}
 In the category of sets, $Sets$, the object $2=\{\bot, \top\}$, which we identify with $\{0,1\}$,  plays a special role, and in the category $\M$ it also plays a special role.  In $\M$ we endow $2$ with the discrete $\sigma$-algebra and     
 a $\M$ arrow  $X \stackrel{f}{\longrightarrow} 2$ evaluated at $x \in X$ gives a probability measure on $2$ and $f(x)[\{\top\}]$ represents the probability of truth of the predicate $f$ at $x$.  Alternatively,  since $T2 \cong [0,1]$ any $\M$ arrow $X \stackrel{f}{\longrightarrow} 2$ is equivalent to a measurable function $X \stackrel{f}{\longrightarrow} [0,1]$ and consequently we identify any $\M$ arrow to $2$ as a measurable function to $[0,1]$ with the restricted Borel $\sigma$-algebra of $\mathbb{R}$.

\section{Deterministic Quantifiers as Adjoints in $Sets$}  In the deterministic setting the existential and universal quantifiers  can be viewed as the left and right adjoint functors to the substitution functor, with these functors operating on  posets~\cite{LawRoseBook,MM}.
 Here, in the category of sets, the basic concepts are well understood and put into the proper framework their extension to the category $\M$ is possible.  Hence  we first summarize the deterministic quantifiers from the well-known viewpoint of quantifiers as adjoints in the category $Sets$ and then extend it to the category $\M$.  
 
In the category of sets, the set of arrows from $X$ to $2$,  denoted $Sets(X,2)$,  can be given a partial ordering by
defining, for $g,h \in Sets(X,2)$,
\begin{equation*}  
\begin{array}{c}
g \vdash_X h \\
\textrm{if and only if} \\
\forall x \in X \quad g(x) \le h(x)
\end{array}
\end{equation*}
where the ordering on $Sets(X,2)$ is determined by the ordering $\bot<\top$ ($0<1$) on the set $2$.  Thus if $g \vdash_X h$ then, for all $x \in X$, the truth of $g(x)$ implies the truth of $h(x)$.  $Sets(X,2)$ can be interpreted as the set of predicates on $X$ and the relation  $g \vdash_X h$ construed as \emph{the predicate }$h$\emph{ is at least as true as the predicate }$g$.   This partial ordering on $Sets(X,2)$ makes it a poset category with the objects being predicates and at most one arrow between any two objects.

Let $X \stackrel{f}{\longrightarrow} Y$ be an arrow in $Sets$.  Since  precomposition preserves ordering, 
\begin{equation*} 
g \vdash_Y h   \Rightarrow g \circ f \vdash_X   h \circ f
\end{equation*} 
each such arrow $f$ determines a functor between the posets, 
\begin{equation*}
\begin{array}{ccc}
Set(Y,2) &\stackrel{f^*}{\longrightarrow}& Set(X,2) \\
g &\mapsto & g \circ f
\end{array}
\end{equation*}
The quantifiers $\exists_f$ and $\forall_f$ are left and right adjoints, respectively, to this ``substitution'' functor $f^*$. 

If $g$ is a predicate on $X$ and $X \stackrel{f}{\longrightarrow} Y$ then the quantifier $\exists_f$ can be illustrated by  the noncommutative diagram

\begin{equation} \label{quantifierDiagram}
 \begin{tikzpicture}[baseline=(current bounding box.center)]
         \node  (X)  at (0,0)    {$X$};
          \node  (Y)  at  (3,0)    {$Y$};	
	\node   (2)  at  (1.5,-2)      {$2$};
	\draw[->,above] (X) to node {$f$} (Y);
	\draw[->,dashed, right] (Y) to node [xshift=3pt] {$\exists_f g$} (2);

	\draw[->,left] (X) to node {$g$}  (2);
	 \end{tikzpicture}
 \end{equation}
where the dashed arrow $\exists_f g$ is a predicate on $Y$ determined by the predicate $g$ and arrow $f$.  By defining
\begin{equation*}   
(\exists_fg)(y) = \sup_{x \in X} \{ g(x) \, | \, y=f(x) \}
\end{equation*}
it follows
\begin{equation*}  
g \vdash_X \exists_fg \circ f
\end{equation*}
and if $h$ is any predicate on $Y$ such that $g \vdash_X f^*(h)$ then $\exists_fg \vdash_Y h$.  

 Thus knowing the relation $g \vdash_X f^*(h)$ we can \emph{infer} the relation $\exists_f g \vdash_Y h$ and vice versa.  This is a powerful inference tool in classical logic and  expressed by
\begin{equation}  \label{leftAdjointEq}
\begin{picture}(200,30)(-63,-10)
\put(0,0){$g \vdash_X f^*(h)$}
\put(-5,-5){\line(1,0){58}}
\put(1,-15){$\exists_f g \vdash_Y h$}
\put(70,3){\vector(0,-1){13}} 
\put(75,-8){\vector(0,1){13}} 
\end{picture}
\end{equation}
where the vertical arrows indicates that the truth of one relation infers the truth of the other relation.
 
This  adjunction is expressed by 
\begin{equation} \label{leftAdjoint}
 \begin{tikzpicture}[baseline=(current bounding box.center)]
         \node  (X)  at (-2,0)    {$Sets(X,2)$};
         \node  (Y)  at  (2,0)    {$Sets(Y,2)$};	
         \node (ad) at (4,0)   {$\exists_f \dashv f^*$};
	\draw[->, above] ([yshift=2pt] X.east) to node {$\exists_f$} ([yshift=2pt] Y.west);
	\draw[->, below] ([yshift=-2pt] Y.west) to node {$f^*$} ([yshift=-2pt] X.east);
    \end{tikzpicture}
 \end{equation} 

Defining the quantifier $\forall_f$ by 
\begin{equation*} 
(\forall_f g)(y) = \inf_{ x \in X } \{ g(x) \, | \, y=f(x) \}
\end{equation*}
yields the adjunction
 
\begin{equation} \label{rightAdjoint}
 \begin{tikzpicture}[baseline=(current bounding box.center)]
         \node  (X)  at (-2,0)    {$Sets(Y,2)$};
         \node  (Y)  at  (2,0)    {$Sets(X,2)$};	
         \node (ad) at (4,0)   {$f^* \dashv \forall_f$};
	\draw[->, above] ([yshift=2pt] X.east) to node {$f^*$} ([yshift=2pt] Y.west);
	\draw[->, below] ([yshift=-2pt] Y.west) to node {$\forall_f$} ([yshift=-2pt] X.east);
    \end{tikzpicture}
 \end{equation} 
and is summarized by the bijective correspondence
\begin{equation} \label{rightAdjointEq}
\begin{picture}(200,30)(-50,-10)
\put(0,0){$f^*(h) \vdash_X g$}
\put(-5,-5){\line(1,0){58}}
\put(1,-15){$h \vdash_Y \forall_f g$}
\put(70,3){\vector(0,-1){13}} 
\put(75,-8){\vector(0,1){13}} 
\end{picture}
\end{equation}

By generalizing the definitions for the quantifiers $\exists_f$ and $\forall_f$  both of these adjunctions, expressions \ref{leftAdjoint} and \ref{rightAdjoint}, also hold in the category $\M$ where the familiar deterministic concepts are extended to include nondeterminism.

\section{Probabilistic Quantifiers}

\subsection{Partial Order Relations}   

Let $X$ be any object in $\M$.
Given two predicates  $g_1,g_2:X \rightarrow  2$ define
a partial ordering  on $\M(X,2)$ by
\begin{equation} \label{Prelation}
\begin{array}{c}
g_1 \vdash_{X} g_2 \\
\textrm{if and only if} \\
\forall x \in X \quad g_1(x)   \le   g_2(x).
\end{array}
\end{equation}
We  interpret the expression ``$g_1 \vdash_{X} g_2$''  as
\begin{quote}
The predicate  $g_2$ is probabilistically at least as true as $g_1$.
\end{quote}

The set $\M(X,2)$  can be viewed as a poset category.

\subsection{The Probabilistic Existential Quantifier} \label{sec:3}

Suppose $Y$ is a measurable space with a countably generated $\sigma$-algebra and. let $X \stackrel{f}{\longrightarrow} Y$  be any $\M$ arrow.    Since  precomposition preserves ordering on posets we have the ``substitution'' functor
\be  \label{substitution}
\M(TY,2) \stackrel{f^*}{\longrightarrow} \M(X,2)
\ee
defined by $f^*(h)=h \circ f$.  

We would like to construct the left adjoint $\exists_f$ to this substitution functor $f^*$, and consequently we need to construct a universal arrow as shown by the dashed arrow in
\begin{equation}   \label{universalArrowDiagram}
 \begin{tikzpicture}[baseline=(current bounding box.center)]
         \node (g)  at (-2,0)   {$X \stackrel{g}{\longrightarrow} [0,1]$};
         \node (ef) at (2,0)     {$X \stackrel{\exists_f g \circ f}{\longrightarrow} [0,1]$};
         \node    (hf)  at  (2,-2)  {$X \stackrel{h \circ f}{ \longrightarrow} [0,1]$};
         \node  (M)   at  (0,-3.)  {in $\mathcal{M}eas_T(X,2)$};
         \node  (N)   at  (6,-3.)  {in $\mathcal{M}eas_T(TY,2)$};

         \node  (e)  at   (5.5,0)  {$TY \stackrel{\exists_fg}{\longrightarrow} [0,1]$};
         \node  (h)  at (5.5,-2)  {$TY \stackrel{h}{\longrightarrow} [0,1]$};
         	\draw[->,dashed,above] (g) to node {$\vdash_X$} (ef);
	\draw[->,left] (g) to node [xshift=-8pt] {$\vdash_X$} (hf);
	\draw[->,right] (ef) to node [xshift=2pt] {$\vdash_X$} (hf);
	\draw[->,right] (e) to node {$\vdash_{TY}$} (h);
 \end{tikzpicture}
 \end{equation} 
so that the bijection
\begin{equation}  \label{leftAdjointEqP}
\begin{picture}(200,22)(-63,-15)
\put(0,0){$g \vdash_X f^*(h)$}
\put(-5,-5){\line(1,0){58}}
\put(1,-15){$\exists_f g \vdash_{TY} h$}
\put(70,3){\vector(0,-1){13}} 
\put(75,-8){\vector(0,1){13}} 
\end{picture}
\end{equation}
holds.
This is the same bijection as expressed in Equation \ref{leftAdjointEq} for the deterministic existential quantifier, only the categories have changed.  
This desired universal arrow, in $\M(X,2)$,
\begin{equation}  \label{desiredAdj}
g \vdash_X  \exists_fg  \circ f 
\end{equation}
expresses the condition that $g(x) \le \exists_fg(f(x))$ for all $x \in X$.  We now proceed to show that such a universal arrow exists.

\begin{lemma} If $(Y,\sa_Y)$ is a measurable space with a countably generated $\sigma$-algebra then $(TY,\sa_{TY})$ has a countable generating set and for each $Q \in TY$ the set $\{Q\}$ is measurable.
\end{lemma}
\begin{proof}  If $\mathcal{G}$ be a countable generating set for $\sa_{Y}$, then $\mathcal{H} =\{\{ev_{B}^{-1}(r,s)\}$, for $B \in \mathcal{G}$ and $r,s \in \mathbb{Q}$ with $(r,s) \subset [0,1]$ is a countable generating set for $TY$, and
\begin{equation*}
\{Q\} = \cap_{B \in \mathcal{H}} B^*  \quad \textrm{where }B^* = \left\{ \begin{array}{ll} B & \textrm{if }Q \in B \\ B^c & \textrm{if }Q \not \in B \end{array} \right.
\end{equation*}

\end{proof}

\begin{lemma}
Suppose $X$ and  $Y$ are measurable spaces with $Y$ having a countably generated $\sigma$-algebra.  Let $X \stackrel{f}{\longrightarrow} Y$ and $X \stackrel{g}{\longrightarrow} 2$ be two $\M$ arrows with $f$ surjective. Then there exists a $\vdash_{TY}$-minimal object $\exists_fg$ in the poset $\M(TY,2)$ satisfying
\begin{equation}  \label{conditionDesired}
g \vdash_{TY} \exists_fg \circ f
\end{equation} 
\end{lemma}
\begin{proof} First note that if $\sa_Y$ has a countable generating set then, by Lemma 1, $\sa_{TY}$ has a countable generating set, say  $\mathcal{G}' $.   Then $\mathcal{G}= \mathcal{G}'  \cup \{ B^c \, | \, B \in \mathcal{G}'  \}$ where $B^c$ is the complement of $B$ in $Y$ is a countable generating set for $\sa_{TY}$ also.  

To prove the lemma we do it in two steps, first assuming the measurable function $g$ is simple, and then letting $g$ be an arbitrary measurable function in $\M(X,2)$.

Step 1:  Suppose $g$ is a simple function with standard representation 
\begin{equation*}
g = \sum_{i=1}^n a_i \chi_{A_i}
\end{equation*}
where $\{A_i\}_{i=1}^n$ is a family of disjoint measurable sets of $\sa_X$.  
Define $B_i$ by 
\begin{equation*}
B_i = \displaystyle{ \bigcap_{\tiny{\begin{array}{cc}B \in \G \\  A_i \subset f^{-1}(B) \end{array}}} }B 
\end{equation*}
Thus $B_i$ is the smallest measurable set in $TY$ whose preimage under $f$ contains $A_i$.
Then for each $i=1,\ldots,n$ the function $h_i = a_i \chi_{B_i}$ is a measurable function on $TY$ and consequently
\begin{equation}
h = \displaystyle{ \sup_{i=1,\ldots,n}}  \{h_i\}
\end{equation}
is a measurable function on $TY$ and satisfies $g(x) \le h(f(x))$ for all $x \in X$.  Moreover we claim that $h$ is the minimal object in $\M(TY,2)$ under the relation $\vdash_{TY}$ satisfying the Condition \ref{conditionDesired} when $g$ is a simple function.  To observe this note that for any $Q \in TY$ such that $\chi_{B_i}(Q) = 1$ there exists an $x \in A_i$ 
such that $f(x)=Q \in B_i$. If there was no such $x \in A_i$ then, since Lemma 4.1 shows $\{Q\} \in \sa_{TY}$, the set $B_i^* = B_i \cap \{Q\}^c$ would be a measurable set in $\sa_{TY}$ with
 $B_i^*$ strictly smaller than $B_i$ satisfying $A_i \subseteq f^{-1}(B_i^*)$ contradicting the fact that $B_i$ is, by construction, the smallest measurable set of $\sa_{TY}$ whose preimage under $f$ contains $A_i$.
Thus it is necessarily the case that there exists an $x\in A_i$ such that $f(x)=Q$ and consequently it is required that 
\begin{equation} \label{minCondition}
a_i = g(x) \le h(f(x))=h(Q)
\end{equation}
so $h(Q)$ must be at least $sup\{ h_i\}$.  Since $h$ is chosen to be the least upper bound it clearly satisfies Condition \ref{minCondition} for all $Q \in TY$ and hence the object $h=\exists_fg \in \M(TY,2)$ satisfies the desired condition \ref{conditionDesired}.

Step 2: Since $g$ is $\mathcal{B}_{[0,1]}$-measurable there exists a sequence $\{\psi_n\}_n$ of measurable simple functions satisfying $\psi_n(x) \rightarrow g(x)$ for each $x \in X$.  By step 1, for each $n$ there exists a $\sa_{TY}$-measurable function $h_n:TY \rightarrow [0,1]$, defined by $h_n \stackrel{\triangle}{=} \displaystyle{ \sup_{i=1,\ldots,n}}  \{h_i\}$,  such that $\psi_n  \vdash_X  h_n \circ f$ with $h_n$ the $\vdash_{TY}$-minimal such object in $\M(TY,2)$.  Next, define
\begin{equation*}
\exists_fg=  \lim_{n \rightarrow \infty}  \sup_{i=1,\ldots,n}h_i = \inf_n \sup_{i=1,\ldots,n}h_i
\end{equation*}
which is $\sa_{TY}$-measurable.
From $h_n(f(x) \ge \psi_n(x) \rightarrow g(x)$, it follows that \mbox{$\exists_fg(f(x)) \ge g(x)$} for all $x\in X$, or equivalently $g \vdash_X \exists_fg \circ f$
 and since $\exists_fg$ is the infimum of the sequence $h_n$ of functions satisfying \mbox{$h_n(f(x) \ge \psi_n(x) \rightarrow g(x)$}  it follows that $\exists_fg$ is the $\vdash_{TY}$-minimal object  in the poset $\M(TY,2)$ satisfying the Condition \ref{conditionDesired}.
\end{proof}

\begin{lemma}  \label{metricLemma}  If the space $Y$ has a countably generated $\sigma$-algebra then the topology $\tau$ generated by the metric
defined on $TY$, for $R,Q \in TY$  by
\be  \label{metric}
d_{TY}(R,Q) = \displaystyle{ \sup_{ B \in \Sigma_Y} } \{ |R(B)-Q(B)|  \}
\ee
has the Borel algebra $\mathcal{B}(\tau) $ coinciding with $\sa_{TY}$.
\end{lemma}
\begin{proof}  Let $\mathcal{G}$ denote a countable generating set for $\sa_Y$. 
Then the $\sigma$-algebra $\sa_{TY}$ is countably generated by the sets $ev_B^{-1}( r,s)$ for $r,s \in \mathbb{Q}$, with $(r,s) \subset [0,1]$, and for $B \in \mathcal{G}$, where $ev_B(\cdot) = Id(\cdot)[B]: TY \rightarrow [0,1]$.  We show each generating element $ev_B^{-1}(r,s) \in \mathcal{B}(\tau)$.
  Around each point $R \in ev_B^{-1}(r,s)$ we can put an open ball contained within that set.  Since
\begin{equation*}
 ev_B^{-1}(r,s) = \{ R \in TY \,| \, R(B) \in (r,s) \}
\end{equation*}
so if $Q \in ev_B^{-1}(r,s)$ with say, $Q(B) = u\in (r,s)$ then choosing $\delta = min(|r-u|, |s-u|)$ gives  $S_{\delta}(Q) \subset  ev_B^{-1}(r,s)$.    Taking the union of all such  balls around each point in $ev_B^{-1}(r,s)$ gives the union of open sets
\begin{equation*}
\displaystyle{ \bigcup_{ \tiny{ R \in ev_B^{-1}(r,s)}}} S_{\delta(R)}(R)
\end{equation*}
where $ \delta(R)=min( | r- R(B)|, |s-R(B)|)$, 
and is consequently an element of $\mathcal{B}(\tau)$.  Thus  each
\begin{equation*}
ev_B^{-1}(r,s) \in \mathcal{B}(\tau)
\end{equation*}
and since these generate $\sa_{TY}$ it follows $\sa_{TY} \subseteq \mathcal{B}(\tau)$.

Now we show the reverse inclusion $\mathcal{B}(\tau) \subseteq \sa_{TY}$.   Consider the open ball
\begin{equation*}
S_{\epsilon}(R) = \{Q \, | \, d_{TY}(R,Q)= \displaystyle{ \sup_{B \in \sa_Y}} \{ |R(B) - Q(B)| \} < \epsilon\}.
\end{equation*}
We have, for the countable generating set $\mathcal{G}$, 
\begin{equation*}
 S_{\epsilon}(R) = \displaystyle{ \bigcap_{ B \in \mathcal{G}}} ev_{B}^{-1}( R(B) - \epsilon, R(B) + \epsilon) \in \sa_{TY}.
\end{equation*}
Thus we conclude $\mathcal{B}(\tau) \subset \sa_{TY}$.
\end{proof}

\begin{lemma} \label{Dudley:Thm} Let $Y$ be a countably generated measurable space.  Consider the measurable space $(TY,\sa_{TY})$, and let $A$ be a subset of $TY$ (not necessarily in $\sa_{TY}$).  Let $h$ be a real valued-function on $A$ measurable for $\sa_A$, the subspace $\sigma$-algebra.  Then $h$ can be extended by zero to a real-valued function on all of $TY$, measurable for $\sa_{TY}$.
\end{lemma}
\begin{proof}  This is  \cite[Theorem 4.2.5]{Dudley} where we have used the result of Lemma \ref{metricLemma} so that the topological space induced by the metric $d$ has the Borel algebra which coincides with the induced $\sigma$-algebra $\sa_{TY}$.
\end{proof}

Let $(f(X),\sa_{f(X)})$ denote the image of an $\M$ arrow $X \stackrel{f}{\longrightarrow} Y$ which is not necessarily surjective, with $f(X)$ having the subspace $\sigma$-algebra.  By Lemma 4.2 there exists a $\vdash_{f(X)}$-minimal object $\exists_fg|_{f(X)}$ such that $g \vdash_{f(X)} \exists_fg|_{f(X} \circ f$.
Define the extension of $\exists_fg|_{f(X)}$  by zero using Lemma~\ref{Dudley:Thm}
\begin{equation}   \label{zeroExtension}
 \begin{tikzpicture}[baseline=(current bounding box.center)]
        \node (X)  at (2,0)     {$X$};
         \node (fX2) at (5,0)   {$(f(X),\sa_{f(X)})$};
         \node    (01)  at  (5,-2)  {$[0,1]$};
         \node  (TY)  at  (8,0)  {$TY$};
         \node  (M)   at  (10.5,-1)  {in $\mathcal{M}eas$};
         \draw[->>,above] (X) to node {$f$} (fX2);
         \draw[->,left] (X) to node [xshift=-3pt] {$g$} (01);
	\draw[->,right] (fX2) to node {$\exists_fg|_{f(X)}$} (01);
	\draw[>->,above]  (fX2) to node {$\iota$} (TY);
	\draw[->,out=270,in=15,right] (TY) to node [xshift=8pt,yshift=-2pt] {$\exists_f g$} (01);
 \end{tikzpicture}
 \end{equation} 
 where $\iota$ is the inclusion mapping.  The measurable function $\exists_fg$ is the $\vdash_{TY}$-minimal object  extending $\exists_fg|_{f(X)}$ to $TY$. This establishes the universal arrow $g  \vdash_{X} \exists_fg \circ f$ in Diagram  \ref{universalArrowDiagram} and hence the bijective correspondence \ref{leftAdjointEqP}.

 If $Y$ a \textit{countable space} 
 then, using the fact that for each probability measure $Q$ on $Y$ the singleton $\{Q\} \in \sa_{TY}$, 
 the expression $\exists_{f} g$  reduces to 
\be \label{existsdef}
\exists_{f} g(Q) = \sup_{ x \in X} \{ g(x) \, | \, Q = f (x) \}.
\ee
where we use the fact $\sup \emptyset =0$.  However if $Y$ is uncountable then such a simple expression need not exist.

Given $g_1 \vdash_X g_2$ it is clear that $\exists_f g_1 \vdash_{TY} \exists_f g_2$ so that
 any $\M$ arrow $X \stackrel{f}{\longrightarrow} Y$  determines a functor of posets
\begin{equation} 
 \begin{tikzpicture}[baseline=(current bounding box.center)]
         \node (X) at (0,0)     {$\M(X,2)$};
         \node (TY) at (4,0)     {$\M(TY,2)$};
	\draw[->, above] ([yshift=2pt] X.east) to node {$\exists_f$} ([yshift=2pt] TY.west);
 \end{tikzpicture}
 \end{equation} 
given by the above construction for $\exists_f$.

The existence of the universal arrow $g \vdash_X \exists_fg \circ f$ for each $g \in \M(X,2)$ and the functoriality of $\exists_f$ establish

\begin{thm} \label{existence:Thm}    Let $Y$ be a measurable space with  a countably generated $\sigma$-algebra. For any  $\M$ arrow $X \stackrel{f}{\longrightarrow} Y$   the adjunction
\begin{equation*}   
 \begin{tikzpicture}[baseline=(current bounding box.center)]
         \node  (X)  at (-3.5,0)    {$\M(X,2)$};
          \node  (Y)  at  (1,0)    {$\M(TY,2)$};	
          \node  (ad)  at  (4,0)  {$ \exists_{f} \dashv f^* $};
	\draw[->, above] ([yshift=2pt] X.east) to node {$\exists_f$} ([yshift=2pt] Y.west);
	\draw[->, below] ([yshift=-2pt] Y.west) to node {$f^*$} ([yshift=-2pt] X.east);
	 \end{tikzpicture}
 \end{equation*} 
holds.
\end{thm}

This theorem could equally well be expressed as 
\begin{equation*} 
 \begin{tikzpicture}[baseline=(current bounding box.center)]
         \node  (X)  at (-4,0)    {$\mathcal{M}eas(X,[0,1])$};
          \node  (Y)  at  (1,0)    {$\mathcal{M}eas(TY,[0,1])$};	
          \node  (ad)  at  (4,0)  {$ \exists_{f} \dashv f^* $};
	\draw[->, above] ([yshift=2pt] X.east) to node {$\exists_f$} ([yshift=2pt] Y.west);
	\draw[->, below] ([yshift=-2pt] Y.west) to node {$f^*$} ([yshift=-2pt] X.east);
	 \end{tikzpicture}
 \end{equation*} 
 with the partial orderings defined on the sets of measurable functions making them poset categories.
 By consideration of composites of existential quantifiers the following result suggest why the formulation in terms of $\M$ is appropriate.

\subsection{Composites of Existential Quantifiers}  

Let $X$ be a measurable space and $Y$ and $Z$ be measurable spaces having countably generated $\sigma$-algebras. Consider the $\M$ diagram
\begin{equation*} 
 \begin{tikzpicture}[baseline=(current bounding box.center)]
         \node  (X)  at (0,0)    {$X$};
          \node  (Y)  at  (2.,0)    {$Y$};	
          \node  (Z)  at  (4,0)  {$Z$};
	\draw[->, above] (X) to node {$f$} (Y);
	\draw[->, above] (Y) to node {$g$} (Z);
	\draw[->,out=-30,in=210,looseness=.5,below] (X) to node {$g \bullet f$} (Z);
	 \end{tikzpicture}
 \end{equation*} 
The three meaurable functions $f,g$,and $\mu_Z$ (viewing the $\M$ arrows $f$ and $g$ as measurable functions) give the sequence of adjunctions
\begin{equation*} 
 \begin{tikzpicture}[baseline=(current bounding box.center)]
         \node  (X)  at (-4.5,0)    {$\M(X,2)$};
          \node  (TY)  at  (-.25,0)    {$\M(TY,2)$};	
          \node  (T2Z)  at  (4.5,0)  {$\M(T^2Z,2)$};
          \node (TZ) at (9,0)  {$\M(TZ,2)$};
	\draw[->, above] ([yshift=2pt] X.east) to node [yshift=-2pt] {$\exists_f$} ([yshift=2pt] TY.west);
	\draw[->, above] ([yshift=2pt] TY.east) to node {$\exists_{Tg}$} ([yshift=2pt] T2Z.west);
	\draw[->,above] ([yshift=2pt] T2Z.east) to node {$\exists_{\mu_Z}$} ([yshift=2pt] TZ.west);
	\draw[->,below] ([yshift=-2pt] TZ.west)  to node {$\mu_Z^*$} ([yshift=-2pt] T2Z.east);
	\draw[->,below] ([yshift=-2pt] T2Z.west)  to node {$(Tg)^*$} ([yshift=-2pt] TY.east);
	\draw[->,below] ([yshift=-2pt] TY.west)  to node [yshift=1pt] {$f^*$} ([yshift=-2pt] X.east);

	 \end{tikzpicture}
 \end{equation*} 
between posets.  On the other hand we also have the adjunction
\begin{equation*} 
 \begin{tikzpicture}[baseline=(current bounding box.center)]
         \node  (X)  at (-1.5,0)    {$\M(X,2)$};
          \node (TZ) at (4,0)  {$\M(TZ,2)$};
	\draw[->, above] ([yshift=2pt] X.east) to node [yshift=-2pt] {$\exists_{g \bullet f}$} ([yshift=2pt] TZ.west);
	\draw[->,below] ([yshift=-2pt] TZ.west)  to node {$(g \bullet f)^*$} ([yshift=-2pt] X.east);
	 \end{tikzpicture}
 \end{equation*} 
 where we view $g \bullet f$ as a measurable function $X \rightarrow TZ$.
Given the two adjunctions,  $\exists_{g \bullet f} \dashv (g \bullet f)^*$ and  $\exists_{\mu_Z} \circ \exists_{Tg} \circ \exists_f \dashv (\mu_Z \circ Tg \circ f)^*$, and by definition of composition in $\M$, $(g \bullet f)^* = (\mu_Z \circ Tg \circ f)^*$  it follows the functors $\exists_{g \bullet f}$ and $\exists_{\mu_Z} \circ \exists_{Tg} \circ \exists_f$ are both left adjuncts of $(g \bullet f)^*$.  In general, adjuncts to a given functor are defined up to a natural isomorphism. However since the categories are posets  it follows that 
\begin{equation}   \label{composites}
\exists_{g \bullet f} = \exists_{\mu_Z} \circ \exists_{Tg} \circ \exists_f
\end{equation}
or equivalently, $\exists_{\mu_Z \circ Tg \circ f} = \exists_{\mu_Z} \circ \exists_{Tg} \circ \exists_f$.

\subsection{The Probabilistic Universal Quantifier} 
Let $X \stackrel{f}{\longrightarrow} Y$ be a $\M$ arrow with $Y$ a countably generated measurable space.
To construct a right adjoint to the substitution functor given in Diagram \ref{substitution} we need to construct a universal arrow as shown by the dashed arrow in Diagram \ref{universalArrowDiagram2}

\begin{equation}   \label{universalArrowDiagram2}
 \begin{tikzpicture}[baseline=(current bounding box.center)]
         \node (gf)  at (-2,0)   {$TX \stackrel{\forall_fg \circ f}{\longrightarrow} [0,1]$};
         \node (g) at (2,0)     {$TX \stackrel{g}{\longrightarrow} [0,1]$};
         \node    (hf)  at  (-2,-2)  {$TX \stackrel{h \circ f}{ \longrightarrow} [0,1]$};
         \node  (M)   at  (0,-3)  {in $\mathcal{M}eas_T(X,2)$};
         \node  (N)   at  (-5,-3.)  {in $\mathcal{M}eas_T(TY,2)$};

         \node  (forallfg)  at   (-5.5,0)  {$TY \stackrel{\forall_fg}{\longrightarrow} [0,1]$};
         \node  (h)  at (-5.5,-2)  {$TY \stackrel{h}{\longrightarrow} [0,1]$};
         	\draw[->,dashed,above] (gf) to node {$\vdash_{X}$} (g);
	\draw[->,below] (hf) to node [xshift=3pt] {$\vdash_{X}$} (g);
	\draw[->,right] (hf) to node [xshift=2pt] {$\vdash_{X}$} (gf);
	\draw[->,right] (h) to node {$\vdash_{TY}$} (forallfg);
 \end{tikzpicture}
 \end{equation} 
so that the bijection
\begin{equation}  \label{Universalbijection}
\begin{picture}(200,22)(-63,-10)
\put(0,0){$h \vdash_{TY} \forall_fg$}
\put(-5,-5){\line(1,0){58}}
\put(1,-15){$h \circ f \vdash_{X} g$}
\put(70,3){\vector(0,-1){13}} 
\put(75,-8){\vector(0,1){13}} 
\end{picture}
\end{equation}
holds. 

The functor $\forall_{f}: \M(X,2)  \rightarrow  \M(TY,2)$
can be defined in an analogous procedure as that for $\exists_f$ by altering Lemma 4.2.

\begin{lemma}
Suppose $X$ and  $Y$ are measurable spaces with $Y$ having a countably generated $\sigma$-algebra.  Let $X \stackrel{f}{\longrightarrow} Y$ and $X \stackrel{g}{\longrightarrow} 2$ be two $\M$ arrows with $f$ surjective. Then there exists a $\vdash_{TY}$-maximal object $\forall_fg$ in the poset $\M(TY,2)$ satisfying
\begin{equation}  \label{conditionDesired2}
\forall_fg \circ f \vdash_{X} g
\end{equation} 
\end{lemma}
\begin{proof} First note that if $\sa_Y$ has a countable generating set then, by Lemma 1, $\sa_{TY}$ has a countable generating set, say  $\mathcal{G}' $.   Then $\mathcal{G}= \mathcal{G}'  \cup \{ B^c \, | \, B \in \mathcal{G}'  \}$ where $B^c$ is the complement of $B$ in $Y$ is a countable generating set for $\sa_{TY}$ also.  

To prove the lemma we do it in two steps, first assuming the measurable function $g$ is simple, and then letting $g$ be an arbitrary measurable function in $\M(X,2)$.

Step 1:  Suppose $g$ is a simple function with standard representation 
\begin{equation*}
g = \sum_{i=1}^n a_i \chi_{A_i}
\end{equation*}
where $\{A_i\}_{i=1}^n$ is a family of disjoint measurable sets of $\sa_X$.  
Define $B_i$ by 
\begin{equation*}
B_i = \displaystyle{ \bigcap_{\tiny{\begin{array}{cc}B \in \G \\  A_i \subset f^{-1}(B) \end{array}}} }  B 
\end{equation*}
Thus $B_i$ is the smallest measurable set in $TY$ whose preimage under $f$ contains $A_i$.
Then for each $i=1,\ldots,n$ the function $h_i = a_i \chi_{B_i}$ is a measurable function on $TY$ and consequently 
\begin{equation}
h = \displaystyle{ \inf_{i=1,\ldots,n}}  \{h_i\}
\end{equation}
is a measurable function on $TY$ and satisfies $h(f(x)) \le g(x)$ for all $x \in X$.  Moreover we claim that $h$ is the maximal object in $\M(TY,2)$ under the relation $\vdash_{TY}$ satisfying the Condition \ref{conditionDesired2} when $g$ is a simple function.  To observe this note that for any $Q \in TY$ such that $\chi_{B_i}(Q) = 1$ there exists an $x \in A_i$ 
such that $f(x)=Q \in B_i$. If there was no such $x \in A_i$ then, since Lemma 4.1 shows $\{Q\} \in \sa_{TY}$, the set $B_i^* = B_i \cap \{Q\}^c$ would be a measurable set in $\sa_{TY}$ with
 $B_i^*$ strictly smaller than $B_i$ satisfying $A_i \subseteq f^{-1}(B_i^*)$ contradicting the fact that $B_i$ is, by construction, the smallest measurable set of $\sa_{TY}$ whose preimage under $f$ contains $A_i$.
Thus it is necessarily the case that there exists an $x\in A_i$ such that $f(x)=Q$ and consequently it is required that 
\begin{equation} \label{maxCondition}
h(f(x))=h(Q) \le g(x) = a_i
\end{equation}
so $h(Q)$ must be at most $\inf\{ h_i\}$.  Since $h$ is chosen to be the greatest lower bound it clearly satisfies Condition \ref{maxCondition} for all $Q \in TY$ and hence the object $\forall_fg = h \in \M(TY,2)$ satisfies the desired Condition \ref{conditionDesired2}.

Step 2: Since $g$ is $\mathcal{B}_{[0,1]}$-measurable there exists a sequence $\{\psi_n\}_n$ of measurable simple functions satisfying $\psi_n(x) \rightarrow g(x)$ for each $x \in X$.  By step 1, for each $n$ there exists a $\sa_{TY}$-measurable function $h_n:TY \rightarrow [0,1]$, defined by $h_n \stackrel{\triangle}{=} \displaystyle{ \inf_{i=1,\ldots,n}}  \{h_i\}$,  such that $ h_n \circ f \vdash_X \psi_n$ with $h_n$ the $\vdash_{TY}$-maximal such object in $\M(TY,2)$.  Next, define
\begin{equation*}
\forall_fg=  \lim_{n \rightarrow \infty}  \inf_{i=1,\ldots,n}h_i = \sup_n \inf_{i=1,\ldots,n}h_i
\end{equation*}
which is $\sa_{TY}$-measurable.
From $h_n(f(x)) \le \psi_n(x) \rightarrow g(x)$, it follows that \mbox{$\forall_fg(f(x)) \le g(x)$} for all $x\in X$, or equivalently $\forall_fg \circ f \vdash_X g$
 and since $\forall_fg$ is the supremum of the sequence $h_n$ of functions satisfying \mbox{$h_n(f(x)) \le \psi_n(x) \rightarrow g(x)$}  it follows that $\forall_fg$ is the $\vdash_{TY}$-maximal object  in the poset $\M(TY,2)$ satisfying the Condition \ref{conditionDesired2}.
\end{proof}

In constructing the right adjoint to the substitution functor, just as in the Diagram \ref{zeroExtension} for the existential quantifier, we have the diagram
\begin{equation*}  
 \begin{tikzpicture}[baseline=(current bounding box.center)]
        \node (X)  at (2,0)     {$X$};
         \node (fX2) at (5,0)   {$(f(X),\sa_{f(X)})$};
         \node    (01)  at  (5,-2)  {$[0,1]$};
         \node  (TY)  at  (8,0)  {$TY$};
         \node  (M)   at  (10.5,-1)  {in $\mathcal{M}eas$};
         \draw[->>,above] (X) to node {$f$} (fX2);
         \draw[->,left] (X) to node [xshift=-3pt] {$g$} (01);
	\draw[->,right] (fX2) to node {$\forall_fg|_{f(X)}$} (01);
	\draw[>->,above]  (fX2) to node {$\iota$} (TY);
	\draw[->,out=270,in=15,right] (TY) to node [xshift=8pt,yshift=-2pt] {$\forall_f g$} (01);
 \end{tikzpicture}
 \end{equation*} 
 where $(f(X),\sa_{f(X)})$ is the image of $f$ with the subspace $\sigma$-algebra, 
 $\iota$ is the inclusion mapping, and to extend the measurable function $\forall_f|_{f(X)}$ from $f(X)$ to $TY$ we take the extension of this measurable function by the constant one.  (In Lemma \ref{Dudley:Thm} the measurable function can extended  by any real-valued constant.)
  The measurable function $\forall_fg$ is the $\vdash_{TY}$-maximal object  extending $\forall_fg|_{f(X)}$ to $TY$. This establishes the universal arrow $g  \vdash_{X} \exists_fg \circ f$ in Diagram  \ref{universalArrowDiagram2} and hence the bijective correspondence \ref{Universalbijection}.

The argument applied to the existential quantifier to show  $\exists_f g$ is the object part of the universal arrow $g \vdash_X \exists_fg \circ f$ applies to show  $\forall_f g$ is the object part of the universal arrow $\forall_fg \circ f \vdash_{X} g$ to obtain
 
\begin{thm} \label{forall:Thm}   Let $Y$ be a measurable space with  a countably generated $\sigma$-algebra.  For any  $\M$ arrow $X \stackrel{f}{\longrightarrow} Y$  the adjunction
\begin{equation*}  
 \begin{tikzpicture}[baseline=(current bounding box.center)]
         \node  (X)  at (-2,0)    {$\M(TY,2)$};
          \node  (Y)  at  (2,0)    {$\M(X,2)$};	
          \node  (ad)  at  (5,0)  {$ f^* \dashv \forall_f$};
	\draw[->, above] ([yshift=2pt] X.east) to node {$f^*$} ([yshift=2pt] Y.west);
	\draw[->, below] ([yshift=-2pt] Y.west) to node {$\forall_f$} ([yshift=-2pt] X.east);
	 \end{tikzpicture}
 \end{equation*} 
holds.
\end{thm}

If  $Y$ is a countable space then the universal quantifier reduces to
\begin{equation*}
\forall_{f} g(Q) = \inf_{ x \in X} \{ g(x) \, | \, Q = f (x) \}
\end{equation*}
where we use the fact $\inf \emptyset = 1$.

The composite of universal quantifiers leads to the result, analogous to Equation \ref{composites}, 
\begin{equation*}
\forall_{g \bullet f} = \forall_{\mu_Z} \circ \forall_{Tg} \circ \forall_f.
\end{equation*}

\subsection{Continuity}

A special case of interest  occurs  when $X$ also has a  countably generated $\sigma$-algebra.  
Suppose this is the case and we have the  $\M$ arrow $X \stackrel{f}{\longrightarrow} Y$.   Define $f^{\sharp}$ by the commutativity of the diagram
\begin{equation}  \label{defLift}  
 \begin{tikzpicture}[baseline=(current bounding box.center)]
         \node  (W)  at (0,2)    {$X$};
         \node (TW) at (-2,0)   {$TX$};
          \node  (Y)  at  (-0,0)    {$Y$};	
          \node  (ad)  at  (3.5,1)  {in $\M$};
	\draw[->, left] (TW) to node [xshift=-2pt] {$Id$} (W);
	\draw[->, right] (W) to node [xshift=2pt] {$f$} (Y);
	\draw[->,above] (TW) to node {$f^{\sharp}$} (Y);
	 \end{tikzpicture}
 \end{equation} 
  where we recall that the $\M$ mapping $Id$ is the measurable identity function \mbox{$TX \rightarrow TX$} giving rise  to the \mbox{$\M$} arrow \mbox{$TX \rightarrow X$}.  Thus $f$ is the lift of $f^{\sharp}$ and
\begin{equation*}
f^{\sharp}(P) = (f \bullet Id)(P) = \int_{x \in X} f(x) \, dP = f \bullet P.
\end{equation*}

\begin{thm} \label{continuity} Let $f$ and $f^{\sharp}$ be defined in Diagram \ref{defLift}.  Then $f^{\sharp}$ is  continuous with respect to the  metrics $d_{TX}$ and $d_{TY}$ whose forms are specified by Equation \ref{metric}. 
\end{thm} 
\begin{proof}  Let $\{ P_i\}_{i=1}^{\infty}  \rightarrow P$ be a convergent sequence in $TX$ with respect to the metric $d_{TX}$.  Observe that convergence in the metric of $TX$ is precisely uniform convergence with respect to the measurable sets $A \in \sa_X$.   Let $||\cdot ||$ denote the total variation of a signed measure.   We will use the fact that $\lim_{i \rightarrow \infty} ||P_i - P || = 0$ holds if and only if $P_i(A)$ converges to $P(A)$ uniformly in $A$. 

For each $B \in \sa_Y$ by definition
\begin{equation*}
f^{\sharp}(P_i) = \int_{x \in X} f(x)[B] \, dP_i
\end{equation*}
and, in general, for any $X$ measurable function $h$ and signed measure $\mu$ on $X$ 
\begin{equation*}
\left| \int_{x \in X} h(x) \, d\mu \right| \le ||h||_{\infty}  ||\mu||
\end{equation*}
where $||h||_{\infty} = \sup_{x \in X}\{|h(x)|\}$.
Taking $h(x) = \hat{f}(x)[B]$ and $\mu = P_i - P$, and using the observation that  $||\hat{f}(x)[B]||_{ \infty} \le 1$ for all $B \in \sa_{Y}$ it follows that 
\begin{equation*}
\forall B \in \sa_Y \quad \left| \int_{x \in X} f(x)[B] \, dP_i  -  \int_{x \in X} f(x)[B] \, dP \right| \le  ||P_i - P|| 
\end{equation*}
Taking the supremum with respect to $B \in \sa_Y$  of this inequality yields
\begin{equation*}
d_{TY}(f^{\sharp}(P_i),f^{\sharp}(P)) \le ||P_i - P||  
\end{equation*}
which implies, by the uniform convergence of $\{P_i\}_i$ which is equivalent to $||P_i - P|| \rightarrow 0$, that $\{f^{\sharp}(P_i)\}_i$ converges in the metric $TY$.       
\end{proof}

When $X$ and $Y$ are finite spaces, which implies each space is  isomorphic to a finite set $\{x_1,x_2,\ldots,x_{m}\}$ and $\{y_1,y_2,\ldots,y_n\}$ with the discrete $\sigma$-algebra, the computation of the quantifiers $(\exists_{f^{\sharp}} g)(Q)$ and $(\forall_{f^{\sharp}} g)(Q)$ 
with respect to the adjunctions
\begin{equation*} 
 \begin{tikzpicture}[baseline=(current bounding box.center)]
         \node  (X)  at (-2.5,0)    {$\M(TY,2)$};
          \node  (Y)  at  (1.,0)    {$\M(TX,2)$};	
          \node  (ad)  at  (-1.5,-1.2)  {$ {f^{\sharp}}^* \dashv \forall_{f^{\sharp}}$};
	\draw[->, above] ([yshift=2pt] X.east) to node {${f^{\sharp}}^*$} ([yshift=2pt] Y.west);
	\draw[->, below] ([yshift=-2pt] Y.west) to node {$\forall_{f^{\sharp}}$} ([yshift=-2pt] X.east);
	
         \node  (X)  at (4.5,0)    {$\M(TX,2)$};
          \node  (Y)  at  (8,0)    {$\M(TY,2)$};	
          \node  (ad)  at  (6,-1.2)  {$ \exists_{f^{\sharp}} \dashv {f^{\sharp}}^* $};
	\draw[->, above] ([yshift=2pt] X.east) to node {$\exists_{f^{\sharp}}$} ([yshift=2pt] Y.west);
	\draw[->, below] ([yshift=-2pt] Y.west) to node {${f^{\sharp}}^*$} ([yshift=-2pt] X.east);
 
 \end{tikzpicture}
 \end{equation*} 

reduce to solving a linear programming problem.
\begin{example}  \label{existentialExample}
Suppose $X=\{x_1,x_2,x_3\}$,  $Y=\{y_1,y_2\}$,  $g=\frac{1}{2} \chi_{\{x_1\}} + \frac{3}{5} \chi_{\{x_2\}} + \frac{9}{10} \chi_{\{x_3\}}$, which induces the relation $\hat{g}(P)=\int_X g \, dP$ on $TX$, and $X \stackrel{f}{\longrightarrow} Y$ is given by the Table~\ref{fconditional}.
\begin{table}[h]
\begin{center}
\begin{tabular}{| c | c | c |}  \hline
$f$ & $\{y_1\}$ & $\{ y_2 \}$ \\ \hline
$x_1$ & $1$ & $0$ \\ \hline
$x_2$ & $\frac{1}{2}$ & $\frac{1}{2}$ \\ \hline
$x_3$ & $\frac{3}{10}$ & $\frac{7}{10}$ \\ \hline
\end{tabular}
\caption[]{The conditional probability $f$.}
\label{fconditional}
\end{center}
\end{table}

Let $P = \sum_{i=1}^3 p_i \delta_{x_i}$.  For $\alpha \in [0,1]$ and $Q=  \alpha \delta_{y_1} + (1- \alpha) \delta_{y_2}$ the computation of the predicate $(\forall_{f^{\sharp}} g)$ at $Q$  is given by solving
\begin{equation*}  
\begin{array}{l}
\displaystyle{ \min_{p_i}} \, \, \frac{1}{2} p_1 + \frac{3}{5} p_2 + \frac{9}{10} p_3 \\
\textrm{subject to}  \\
\begin{array}{lclcccccccc}
Q(\{y_1\}) &=&  \int_X f(x)[ \{y_1\}] \, dP  & \Rightarrow &    \alpha &=&  p_1 &+& \frac{1}{2} p_2 &+& \frac{3}{10} p_3  \\
Q(\{y_2\}) &=& \int_X f(x)[\{y_2\}] \, dP   & \Rightarrow &    1-\alpha &=& 0  &+& \frac{1}{2} p_2 &+& \frac{7}{10} p_3 
\end{array}  \\
\, \, p_i \ge 0 \quad \textrm{for }i=1,2,3
\end{array}
\end{equation*}
For $\alpha=.7$ it follows $(\forall_{f^{\sharp}} g)(Q) =\frac{14}{25}$  with the argmin $P=\frac{2}{5} \delta_{x_1} + \frac{3}{5} \delta_{x_2} + 0 \delta_{x_3}$.  By changing the problem to a maximization problem one computes $(\exists_{f^{\sharp}} g)(Q) = \frac{47}{70}$  with argmax $P= \frac{4}{7} \delta_{x_1} + 0 \delta_{x_2} +  \frac{3}{7} \delta_{x_3}$.
\end{example}

\vspace{.1in}
\textit{Acknowledgements.}
The author would like to thank Jared Culbertson for many fruitful conversations and providing clarifications regarding these ideas. This work was partially supported by AFOSR, for which the author is grateful.  

\bibliographystyle{amsplain}

\vspace{.4in}
\small{
\begin{flushleft}
Kirk Sturtz \\
Universal Mathematics \\
Vandalia, OH 45377 \\
\email{kirksturtz@UniversalMath.com} \\
\end{flushleft}
 }
\end{document}